\documentclass[a4paper]{article}
\usepackage{amsmath, amsfonts}
\usepackage{amssymb, latexsym}
\usepackage{amsthm}
\usepackage{xypic}
\usepackage{stmaryrd}

\usepackage[
colorlinks=true,       
linkcolor=blue,          
citecolor=blue,        
plainpages=false,      
pdfpagelabels,
]{hyperref}

\usepackage[shortlabels]{enumitem}
\setlist[enumerate]{leftmargin=*,align=left,labelindent=\parindent}

\newcommand{\Bin}{\left\{ 0,1 \right\}}
\newcommand{\Pow}[1]{\mathcal{P}(#1)}
\DeclareMathOperator{\amp}{\,\&\,}
\DeclareMathOperator{\imp}{\,\rightarrow\,}
\DeclareMathOperator{\defeqiv}{\stackrel{\textup{def}}{\iff}}
\DeclareMathOperator{\defeql}{\stackrel{\textup{def}}{\  =\  }}
\newcommand{\dotminus}{\mathbin{\ooalign{\hss\raise.5ex\hbox{$\cdot$}\hss\crcr$-$}}}
\DeclareMathOperator{\rlz}{\Vdash}
\DeclareMathOperator{\cov}{\mathbin{\lhd}}
\DeclareMathOperator{\bcov}{\mathrel{\blacktriangleleft}}
\newcommand{\ext}[1]{#1_{\preccurlyeq}}
\newcommand{\Nat}{\mathbb{N}}
\newcommand{\FSeq}{\Nat^{*}}
\newcommand{\BSeq}{\Bin^{*}}
\newcommand{\FBaire}{\mathcal{B}}
\newcommand{\FNat}{\mathcal{N}}
\newcommand{\PBaire}{\Nat^{\Nat}}
\newcommand{\PCantor}{\Bin^{\Nat}}
\newcommand{\nbf}{K}
\newcommand{\nil}{\langle\, \rangle}
\newcommand{\seq}[1]{\langle#1\rangle}
\newcommand{\prefix}{\preccurlyeq}
\newcommand{\spfix}{\prec}
\newcommand{\iseg}[2]{\overline{#1}#2}
\newcommand{\lh}[1]{\lvert #1 \rvert}
\newcommand{\curry}[2]{\lambda#1.#2}
\newcommand{\ibar}[1]{\mathrm{Bar}(#1)}
\newcommand{\CC}{\mathrm{\textup{AC}_{\omega}}}
\newcommand{\BO}{\mathit{BO}}
\newcommand{\Pt}[1]{\mathrm{Pt}(#1)}
\newcommand{\Cov}{\mathrm{Cov}}
\newcommand{\PiFT}{\textrm{\textup{$\Pi^{0}_{1}$-FAN}}}

\newcommand{\scFT}{\textrm{\textup{sc-FAN}}}
\newcommand{\cBI}{{\textrm{\textup{c-BI}}}}

\newcommand{\UCb}{{\textrm{\ensuremath{\textup{UC}_{\mathbf{B}}}}}}
\newcommand{\SC}{{\textrm{\textup{SC}}}}
\newcommand{\FC}{{\textrm{\textup{FC}}}}

\newcommand{\zero}{0^{\omega}}

\newcommand{\CZF}{{\textrm{\textup{CZF}}}}

\newtheorem{theorem}{Theorem}[section]
\newtheorem{proposition}[theorem]{Proposition}
\newtheorem{lemma}[theorem]{Lemma}
\newtheorem{corollary}[theorem]{Corollary}
\theoremstyle{definition}
\newtheorem{definition}[theorem]{Definition}
\theoremstyle{remark}
\newtheorem{remark}[theorem]{Remark}

\numberwithin{equation}{section}

\title{Formally continuous functions on Baire space}

\author{
  Tatsuji Kawai\\[0.5em]
\normalsize Dipartimento di Matematica, Universit\`{a} di Padova\\
\normalsize via Trieste 63, 35121 Padova, Italy\\
\normalsize\texttt{tatsuji.kawai@math.unipd.it}}

\date{}
\begin{document}
\maketitle
\begin{abstract}
A function from Baire space $\PBaire$ to the natural numbers
$\Nat$ is called formally continuous if it is induced by a morphism
between the corresponding formal spaces. We compare
formal continuity to two other notions of continuity on Baire space
working in Bishop constructive mathematics: one is a function induced
by a Brouwer-operation (i.e.\ inductively defined neighbourhood
function); the other is a function uniformly continuous near every
compact image. We show that formal continuity is equivalent to the
former while it is strictly stronger than the latter.
\medskip

\noindent\textsl{Keywords:} Constructive mathematics, Formal space, Baire space,
Brouwer-operation \\[3pt]
\noindent\textsl{MSC2010:} 03F60; 03F55; 06D22
\end{abstract}

\section{Introduction}\label{sec:Introduction}
In the development of Bishop  constructive mathematics \cite{Bishop-67},
pointwise continuous functions on a
compact metric space need not be uniformly continuous.  Thus, we need
to adopt a stronger notion of continuity in order to avoid Fan Theorem
to which a recursive counter example is known~\cite[Chapter 4, Section
7.6]{ConstMathI}.
In particular, Bishop defined a function on a locally compact metric
space to be \emph{continuous} if it is uniformly continuous on every
compact subset.\footnote{We refer to Bishop \cite[Chapter 4]{Bishop-67}
for terminology for metric spaces. In particular, \emph{compact} means complete and totally bounded.}

This notion works well as long as locally compact metric spaces are
concerned. Moreover, it was shown by Palmgren \cite{PalmgrenLocalicCompletion}
that continuous maps between locally compact metric spaces
are equivalent to morphisms between the corresponding formal
spaces (i.e.\ constructive point-free topologies \cite{Sambin:intuitionistic_formal_space}).
Specifically, there exists a bijective correspondence between
continuous maps between locally compact metric spaces and morphisms
between their localic completions, the latter being a particularly
well-behaved construction of point-free topologies from metric spaces
due to Vickers \cite{LocalicCompletionGenMet}. Then, a function $f \colon X
\to Y$ between complete metric spaces is called \emph{formally
continuous} if it is induced by a morphism between the localic
completions of $X$ and $Y$; see Palmgren \cite[Section 2]{PalmgrenFormalContUnifContNCI}.
Palmgren's result \cite{PalmgrenLocalicCompletion} says that Bishop's continuity and formal continuity
are equivalent for locally compact metric spaces.
 
Subsequently, Palmgren \cite{PalmgrenFormalContUnifContNCI} studied the
relation between formal continuity and continuity in Bishop
constructive analysis in a wider context of complete metric spaces.
In this  context, the following notion of continuity is often used in Bishop
constructive analysis~\cite{BridgesConstFunctAnalysis}.\footnote{The
notion of strongly continuous function is attributed to Bishop by
Bridges \cite{BridgesFunctionSpace}.
Bridges \cite{BridgesConstFunctAnalysis} calls these functions
\emph{continuous} maps; however, in order to avoid a possible
confusion, we adopt the terminology from Troelstra and {van Dalen} \cite[Chapter
7]{ConstMathII}.}  Note that this extends the notion of 
continuity on locally compact metric spaces; see e.g.\
Palmgren \cite[Proposition 1.4]{PalmgrenFormalContUnifContNCI}.
\begin{definition}\label{def:StrongContinuous}
A subset $L$ of a metric space $X$ is a \emph{compact image} if there
exist a compact metric space $K$ and a uniformly continuous function λ$f
\colon K \to X$ with $L = f[K]$. A function $f \colon  X \to Y$
between metric spaces is \emph{strongly continuous} (or uniformly
continuous near every compact image) if for each compact image $L
\subseteq⊆ X$ and each $\varepsilonε> 0$, there is $\delta > 0$ such
that for all $x,u \in X$
\[
  x \in L \amp d_{X}(x,u) < \delta \implies d_{Y}(f(x),f(u)) <
  \varepsilon.
\]
\end{definition}
Palmgren \cite[Theorem 2.3]{PalmgrenFormalContUnifContNCI}
showed that for complete metric spaces, formal continuity implies
strong continuity. He conjectured that strong continuity would be
strictly weaker than formal continuity; in particular, there would be
a model of Bishop constructive analysis in which there is  a strongly
continuous function on Baire space $\PBaire$ which is not formally
continuous. Note that Baire space with its product metric is a
paradigmatic example of a complete metric space which is not locally
compact. 

The aim of this paper is twofold. The first is to answer Palmgren's
conjecture.
This is done by comparing the strength the following two statements:
\begin{enumerate}
  \item Every pointwise continuous function
    $F \colon \PBaire \to \Nat$ is strongly continuous.
  \item Every pointwise continuous function
    $F \colon \PBaire \to \Nat$ is formally continuous.
\end{enumerate}
We show that the first statement follows from Brouwer's Fan Theorem
while the second statement implies decidable Bar Induction. There are
already several models of constructive analysis in which Fan Theorem
holds but decidable Bar Induction fails; see Fourman and Hyland \cite{FourmanHyland}.
In one of such models, we find a strongly continuous function
which is not formally continuous;
see Section \ref{sec:StrongContinuity}.

The second aim is to characterise formally continuous
functions between Baire space and natural numbers.
We show that formally continuous functions are equivalent to functions
induced by Brouwer-operations, the notion which is familiar in intuitionistic
mathematics; see Section \ref{sec:BrouwerContinuity}.
Our result suggests that Brouwer-operations may provide a good notion
of continuity on Baire space in Bishop constructive analysis.

Section \ref{sec:StrongContinuity} and Section
\ref{sec:BrouwerContinuity} are independent and can be read in any
order. Section \ref{sec:FormalContinuity} is a preliminary on formal
Baire space and formally continuous functions.

\subsubsection*{Formal system}
We work in Aczel's constructive set theory $\CZF$, which may serve as
a formal system for Bishop's constructive mathematics. Our basic
reference for $\CZF$ is the note by Aczel and Rathjen \cite{Aczel-Rathjen-Note}.
In Section \ref{sec:FormalContinuity} and Section \ref{sec:StrongContinuity}, we work in $\CZF$ extended with the Regular
Extension Axiom. This axiom is more than sufficient for the definition
of formal Baire space; see Aczel \cite[Section 6]{AspectofTopinCZF20063}.
In Section \ref{sec:BrouwerContinuity}, we work in $\CZF$
extended with Countable Choice ($\CC$) and the axiom asserting
that ``The Brouwer ordinals form a set''. These axioms are discussed in
detail by van~den Berg and Moerdijk \cite[Appendix B]{vandenBerg20121367}.
Here, the class $\BO$ of \emph{Brouwer ordinals} is the smallest class
that is closed under the following clauses:
\begin{enumerate}[({BO}1)]
  \item $* \in \BO$,
  \item $t \in \BO^{\Nat} \implies \sup_{n \in \Nat}(t_n) \in \BO$.
\end{enumerate}
 van~den Berg and Moerdijk \cite{vandenBerg20121367} showed that
the system
$\CZF + \CC +
$``The Brouwer ordinals form a set'' allows us to define formal Baire
space; see also Remark \ref{rem:FBDefined}. Hence, this system seems
to be a minimum setting in which the results of Section
\ref{sec:BrouwerContinuity} can be formalised.

\subsubsection*{Notation}
We adopt the following notation.
The set of finite sequences of natural numbers is denoted by
$\FSeq$, and the set of finite binary sequences is denoted by
$\BSeq$. The letters $k,n,m,N,M$ range over natural numbers $\Nat$,
and $a,b$ range over $\FSeq$.
Greek letters $\alpha,\beta,\gamma,\dots$ range over the
sequences $\PBaire$. The symbol $\zero$ denotes the constant sequence of $0$.

An element of $\FSeq$ of length $n$ is denoted by
$\seq{a_0,\dots,a_{n-1}}$, and the empty sequence is denoted by
$\nil$. The length of $a$ is denoted by $|a|$, and the concatenation
of $a$ and $b$ is denoted by $a * b$.  The concatenation of finite
sequence $a$ followed by a sequence $\alpha$ is denoted by  $a *
\alpha$ so that $\left( \forall n \in \Nat\right) n \geq |a| \imp a *
\alpha(n) = \alpha(n \dotminus |a|)$.  The initial segment of $\alpha$
of length $n$ is denoted by $\overline{\alpha}n$.  Sometimes, we
identify a finite sequence $a$ with a basic open subset of
Baire space with the product topology.  In this case, $\alpha
\in a$ means $\overline{\alpha}|a| = a$.  The relation $a \prefix b$
(or $a \prec b$) means that $a$ is an initial segment (respectively
strict initial segment) of $b$. We often use lambda notation to denote
functions, for example $\zero = \lambda n. 0$.
 
\section{Formally continuous functions on Baire space}\label{sec:FormalContinuity}
We recall the notion of formal Baire space from Fourman and Grayson
\cite[Example 2.6 (2)]{FormalSpace}, and that of formally continuous function from
Palmgren \cite[Section 2]{PalmgrenFormalContUnifContNCI}.
We use the predicative notion of formal space, i.e.\ formal
topology by Sambin \cite{Sambin:intuitionistic_formal_space}.
Our reference for formal topology is Fox \cite{Fox05}.

\begin{definition}\label{def:FormalBaire}
\emph{Formal Baire space} $\FBaire$ is a pair
$\left(\FSeq, \cov_{\FBaire}\right)$, where 
$\cov_{\FBaire} \subseteq \FSeq \times \Pow{\FSeq}$ 
is a relation between $\FSeq$ and the subsets of $\FSeq$
 inductively defined by the following three clauses:
\begin{gather*}
  \frac{a  \in U}{a \cov_{\FBaire} U}\,(\eta) \qquad
  \frac{a \cov_{\FBaire}U}{a * \seq{k} \cov_{\FBaire}
U}\,(\zeta) \qquad
  \frac{\left( \forall n \in \Nat \right) a*\seq{n} \cov_{\FBaire}
   U}{a \cov_{\FBaire} U}\,(\digamma)
\end{gather*}
Note that $a \cov_{\FBaire} U$ if and only if there exists a ``canonical
proof'' of the fact that ``$U$ bars $a$''; see Brouwer
\cite[Section 2]{BrouwerDomainsofFunctions}
or Troelstra and {van Dalen} \cite[Chapter 4, Section 8.18]{ConstMathI}.
\end{definition}

It is well known that $\zeta$-inference can be eliminated in the
following sense; see Troelstra and {van Dalen} \cite[Chapter 4, Exercise 4.8.10]{ConstMathI}.
\begin{lemma} \label{lem:ElimZeta}
  Let $\bcov_{\FBaire}$ be the relation between $\FSeq$ and $\Pow{\FSeq}$
  inductively defined by $\eta$ and $\digamma$-inferences in
  Definition \ref{def:FormalBaire}. Then, 
  \[
    a \cov_{\FBaire} U \iff a \bcov_{\FBaire} \ext{U}
  \]
  for all $a \in \FSeq$ and $U \subseteq \FSeq$, where
  $\ext{U}$ is the closure of $U$ under extension:
  \[
    \ext{U} \defeql \left\{ a \in \FSeq \mid \left(
    \exists b \in U \right) b \preccurlyeq a \right\}.
  \]
\end{lemma}
\begin{proof}
  By induction on $\cov_{\FBaire}$ and $\bcov_{\FBaire}$.
\end{proof}

\begin{definition}
  A \emph{formal point} of $\FBaire$ is a subset $\alpha \subseteq
  \FSeq$ such that
  \begin{enumerate}
    \item\label{PtInhabited} $\left( \exists a \in \FSeq \right) a \in \alpha$;
    \item\label{PtConvergent} $a,b \in \alpha \implies a \prefix b \vee b \prefix a$;
    \item $a \prefix b \in \alpha \implies a \in \alpha$;
    \item $a  \in \alpha \implies \left( \exists n \in \Nat \right)a *
      \langle n \rangle \in \alpha$.
  \end{enumerate}
  The set of formal points of $\FBaire$ is denoted by $\Pt{\FBaire}$.
\end{definition}

By induction on $\cov_{\FBaire}$, one can show  that a subset $\alpha
\subseteq \FSeq$ is a formal point if and only if
$\alpha$ satisfies \ref{PtInhabited} and \ref{PtConvergent} above, and
for each $a \in \FSeq$ and $U \subseteq \FSeq$
\begin{equation}\label{eq:PtSplit}
  a \cov_{\FBaire} U \amp a \in \alpha \implies \left( \exists b \in U
  \right) b \in \alpha.
\end{equation}
Note that we can identify a formal point $\alpha \in \Pt{\FBaire}$
with a sequence $p_{\alpha} \in \PBaire$ defined by
\begin{equation}\label{eq:Falpha}
  p_{\alpha}(n) \defeql a_{\alpha}(n)
\end{equation}
where  $a_{\alpha}$ is a unique $a_{\alpha} \in \alpha$ such that $|a_{\alpha}| = n + 1$.

\begin{definition}
\emph{Formal natural numbers} $\FNat$ is a pair
$\left(\Nat, \cov_{\FNat} \right)$ where 
the relation $\cov_{\FNat} \subseteq \Nat \times \Pow{\Nat}$ is the
membership $\in$.
A formal point of $\FNat$ is just a singleton subset of $\Nat$.
The set of formal points of $\FNat$ is denoted by $\Pt{\FNat}$.
\end{definition}

\begin{definition}
A \emph{formal topology map} from $\FBaire$ to $\FNat$ is a
relation $r \subseteq \FSeq \times \Nat$ such that
\begin{enumerate}
  \item  $\nil \cov_{\FBaire} r^{-}\Nat$,
  \item  $\ext{(r^{-}\left\{ n \right\})} \cap \ext{(r^{-}\left\{ m \right\})}
    \cov_{\FBaire} r^{-}\left\{ l \in \Nat \mid l = n = m \right\}$,
\end{enumerate}
where
  $
  r^{-} D \defeql \left\{ a \in \FSeq \mid \left( \exists n \in D
    \right) a \mathrel{r} n \right\}
  $
for each $D \subseteq \Nat$.
\end{definition}
By the condition \eqref{eq:PtSplit}, it is easy to see that the function $\Pt{r}
\colon \Pt{\FBaire} \to \Pt{\FNat}$ defined by
\[
  \Pt{r}(\alpha) \defeql \left\{ n \in \Nat \mid \left( \exists a \in
    \alpha \right) a \mathrel{r} n \right\}
\]
is a well-defined mapping from $\Pt{\FBaire}$ to $\Pt{\FNat}$.

\begin{definition}\label{def:FRep}
  A function $F \colon \PBaire \to \Nat$ is \emph{formally
  continuous} if there exists a formal topology map $r$
  from $\FBaire$ to $\FNat$ which makes the following
  diagram commute:
  \[
   \xymatrix{
     \PBaire \ar[r]^-{i_{\FBaire}} \ar[d]_{F} & \Pt{\FBaire}
     \ar[d]^{\Pt{r}} \\
     \Nat \ar[r]^-{i_{\FNat}} & \Pt{\FNat}
   }
  \]
  Here, $i_{\FBaire}$ and $i_{\FNat}$ are bijections defined by
  \begin{align*}
    i_{\FBaire}(\alpha) &\defeql \left\{ \iseg{\alpha}{n} \mid n \in \Nat
  \right\},\\
    i_{\FNat}(n) &\defeql \left\{ n \right\}.
  \end{align*}
\end{definition}

\begin{remark}
Palmgren \cite[Section 3]{PalmgrenFormalContUnifContNCI} showed that formal Baire space
is the localic completion of Baire space with the product metric
\begin{equation}\label{eq:BaireMet}
  d(\alpha,\beta)
  \defeql \inf \left\{ 2^{-n} \mid \overline{\alpha} n =
  \overline{\beta} n\right\}.
\end{equation}
Hence, the notion of formally continuous function given in Definition
\ref{def:FRep} is equivalent to Palmgren's corresponding notion in~\cite{PalmgrenFormalContUnifContNCI}.
\end{remark}

\section{Strongly continuous functions}\label{sec:StrongContinuity}
We show that strongly continuous functions on Baire space need not be
formally continuous.

The definition of strongly continuous function simplifies when the
domain is complete; see Troelstra and {van Dalen} \cite[Chapter 7, Section 4.8]{ConstMathI}.
\begin{lemma}
A function $f \colon  X \to Y$ from a complete metric space $X$ to a
metric space $Y$ is strongly continuous if and only if for each
compact subset $K \subseteq⊆ X$ and each $\varepsilonε> 0$, there is
$\delta > 0$ such that for all $x,u \in X$
\[
  x \in K \amp d_{X}(x,u) < \delta \implies d_{Y}(f(x),f(u)) <
  \varepsilon.
\]
\end{lemma}
We focus on the special case where $X$ is Baire space
$\PBaire$ and $Y$ is the discrete space of $\Nat$. In this case,
a strongly continuous function admits a simple characterisation.
To see this, we recall further terminology.

A \emph{spread} is a decidable tree $T \subseteq \FSeq$ such that
\[
  \left( \forall a \in T \right)\left( \exists n \in \Nat \right)
  T(a * \seq{n}).
\]
Here, $T(a * \seq{n})$ means $a * \seq{n} \in T$.
A \emph{fan} is a spread $T$ such that
\[
  \left( \forall a \in T \right)\left( \exists N \in \Nat \right)
  \left( \forall n \in \Nat \right)\left[T(a * \seq{n}) \implies n \leq N \right].
\]
A sequence $\alpha \in \PBaire$ is a \emph{path} in a tree
$T \subseteq \FSeq$, written $\alpha \in T$, if $\left( \forall n \in \Nat
\right) T(\overline{\alpha}n)$.

It is known that every inhabited compact subset of Baire space (with
the metric defined by \eqref{eq:BaireMet}) can be represented by the set of paths of
some fan. Thus, the following is clear.
\begin{lemma}\label{lem:StrongContBaire}\leavevmode
  A function $F \colon \PBaire \to \Nat$ is strongly continuous if and only if
  for each fan $T$, there exists $N \in \Nat$ such that
  \begin{equation*}
    \left( \forall \alpha \in T \right)
    \left( \forall \beta \in \PBaire \right)
    \overline{\alpha} N = \overline{\beta} N 
    \implies F(\alpha) = F(\beta).  
  \end{equation*}
\end{lemma}
We study the strength of the following statement:
\begin{description}
  \item[\SC] Every pointwise continuous function
    $F \colon \PBaire \to \Nat$ is strongly continuous.
\end{description}
We restate $\SC$ in the style of Fan Theorem.

Given a spread $T$, a subset $P \subseteq \FSeq$ is a \emph{bar} of
$T$ if 
\[
  \left( \forall \alpha \in T \right)\left( \exists n \in \Nat \right)
  P(\overline{\alpha}n).
\]
Note that if $T'$ is a sub-spread of $T$ and $P$ is a bar of $T$, then $P$ is
a bar of $T'$.
A bar $P$ of a spread $T$ is \emph{uniform} if 
\[
  \left( \exists N \in \Nat \right)\left( \forall \alpha \in
  T \right)\left( \exists n \leq N \right) P(\overline{\alpha}n).
\]
A subset $P \subseteq \FSeq$ is a \emph{c-set} if there exists a
function $\delta \colon \FSeq \to \Nat$ such that
\[
  \left(\forall a \in \FSeq  \right)
  \left[P(a) \iff
    \left( \forall b \in \FSeq \right) \delta(a) = \delta(a * b)
  \right].
\]
Note that every c-set is \emph{monotone}, i.e.\ closed under extension.
A \emph{c-bar} is c-set that is a bar of the universal spread $\FSeq$.

The principle $\scFT$ is the following statement:
\begin{description}
 \item[$\scFT$] Every c-bar is uniform with respect to every fan.
\end{description}
In other words, $\scFT$ states that for every c-bar $P$ and fan $T$,
there exists $N \in \Nat$ such that
$
\left( \forall \alpha \in T \right) P(\overline{\alpha}N).
$
\begin{proposition}\label{prop:scFTequivSC}
  $\scFT \iff \SC$.
\end{proposition}
\begin{proof}
  ($\Rightarrow$)
  Assume $\scFT$.
  Let $F \colon \PBaire \to \Nat$ be a pointwise continuous function,
  and let $T \subseteq \FSeq$ be a fan. Define a subset
  $P \subseteq \FSeq$ by
  \begin{equation}\label{eq:cBarFromFun}
    P(a) \defeqiv \left( \forall b \in \FSeq \right)
    F(a * \zero) = F(a * b * \zero).
  \end{equation}
  Since $F$ is pointwise continuous,  $P$ is a c-bar.
  By $\scFT$, there exists $N \in \Nat$ such that
    $
    \left( \forall \alpha \in T \right) P(\overline{\alpha}N),
    $
  i.e.\
  \[
    \left( \forall \alpha \in T \right) \left( \forall b \in \FSeq
    \right) F(\overline{\alpha}N * \zero) = F(\overline{\alpha}N *
    b * \zero).
  \]
  Let $\alpha \in T$ and $\beta \in \PBaire$ and suppose that
  $\overline{\alpha}N = \overline{\beta}N$. 
  We have
    $
    \left( \forall b \in \FSeq \right) F(\overline{\alpha}N *
    \zero) = F(\overline{\alpha}N * b * \zero).
    $
  Since $F$ is pointwise continuous, there exists $m \geq N$ such
  that $F(\overline{\alpha}m * \zero) = F(\alpha)$ and
  $F(\overline{\beta}m * \zero) = F(\beta)$. Then,
  \[
    F(\alpha)
    = F(\overline{\alpha}m * \zero)
    = F(\overline{\alpha}N * \zero)
    = F(\overline{\beta}m * \zero)
    = F(\beta).
  \]
  Hence, $F$ is strongly continuous.

  \medskip

\noindent($\Leftarrow$)
Assume $\SC$. Let $P$ be a c-bar and let $T$ be a fan. Then, there exists
$\delta \colon \FSeq \to \Nat$ such that 
  $
  P(a) \iff \left( \forall b \in \FSeq \right) \delta(a) =
  \delta(a * b)
  $
for all $a \in \FSeq$.
Define a function $F \colon \PBaire \to \Nat$ by
\begin{equation}\label{eq:FuncFromcBar}
  F(\alpha) \defeql \max D_{\alpha}
\end{equation}
where
   $ D_{\alpha} \defeql \left\{ n \in \Nat \mid
   \delta(\overline{\alpha}n) \neq
   \delta(\overline{\alpha}(n+1))\right\} \cup \{1\}$.
It is straightforward to show that $F$ is pointwise continuous.
Then $F$ is strongly continuous by $\SC$.  Thus, there exists $N \in
\Nat$ such that 
\[
  \left( \forall \alpha \in T \right)
  \left( \forall \beta \in \PBaire \right)
  \overline{\alpha} N = \overline{\beta} N 
  \implies F(\alpha) = F(\beta).  
\]
Define $M \in \Nat$ by
\[
  M \defeql \max \left\{N, \max \left\{ F(a * \zero) \mid a \in T \amp
  |a| = N \right\}  \right\} + 1.
\]
Let $\alpha \in T$ and $b \in \FSeq$. Since
  $
  M > F(\overline{\alpha}M * \zero) = F(\overline{\alpha}M * b *
  \zero),
  $
we have $\delta(\overline{\alpha}M) = \delta(\overline{\alpha}M * b)$.
Hence $P(\overline{\alpha}M)$. Thus $P$ is a uniform bar of
$T$.
\end{proof}

If $P$ is a c-bar and $T$ is a fan, then $P$ is a 
$\Pi^{0}_{1}$-bar of $T$, i.e.\ there exists a
decidable subset $D \subseteq T \times \FSeq$ such that
\[
  \left( \forall a \in T \right)\Bigl[  P(a) \iff \left( \forall b \in \FSeq
  \right) D(a,b) \Bigr].
\]
Let $\PiFT(T)$ be the following statement about a fan $T$:
\begin{description}
 \item[$\PiFT(T)$] Every $\Pi^{0}_{1}$-bar of $T$ is uniform.
\end{description}
If we let $\PiFT$ stand for $\PiFT(\BSeq)$, then
$\PiFT(T)$ (for any fan $T$) follows from $\PiFT$. This can be proved
by a straightforward modification of the proof of a similar fact about
Fan Theorem; see Troelstra and {van Dalen} \cite[Chapter 4, Section 7.5]{ConstMathI}.
Thus, we have an upper bound of the strength of $\SC$.
\begin{corollary}
  $\PiFT \implies \SC$.
\end{corollary}
\begin{remark}
  It is easy to see that $\SC$ implies the uniform continuity
  principle, which says that every pointwise
  continuous function $F \colon \PCantor \to \Nat$ is uniformly
  continuous (cf.\ Troelstra and {van Dalen} \cite[Chapter 4, Lemma 1.4]{ConstMathI}).
\end{remark}

Next, we consider a similar statement for formal continuity:
\begin{description}
  \item[\FC] Every pointwise continuous function
    $F \colon \PBaire \to \Nat$ is formally continuous.
\end{description}
\begin{remark}
The statement $\FC$ implies $\SC$ by the result of Palmgren \cite[Theorem
2.3]{PalmgrenFormalContUnifContNCI}, where he showed that formal
continuity implies strong continuity.
\end{remark}

We show that $\FC$ is equivalent to the following variant of Bar
Induction introduced in \cite{KawaiUnifContBaire}:
\begin{description}
  \item[\cBI] For any c-bar $P \subseteq \FSeq$ and a subset $Q \subseteq
    \FSeq$, if $P \subseteq Q$ and $Q$ is inductive, then
    $Q(\nil)$.
\end{description}
Here, a subset $Q \subseteq \FSeq$ is \emph{inductive} if
  $
  \left( \forall a \in \FSeq \right) \left[ \left( \forall n \in \Nat
  \right) Q(a * \langle n \rangle) \implies  Q (a) \right].
  $

We can restate $\cBI$ in terms of ``canonical proof'' as follows.
\begin{proposition}\label{prop:cBICanon}
  The following are equivalent:
  \begin{enumerate}
    \item\label{prop:cBICanon1}
      $\cBI$
    \item\label{prop:cBICanon2}
      $\nil \cov_{\FBaire} P$ for every c-bar $P \subseteq \FSeq$.
  \end{enumerate}
\end{proposition}
\begin{proof}
  ($\ref{prop:cBICanon1} \Rightarrow \ref{prop:cBICanon2}$) Assume
  $\cBI$. Let $P \subseteq \FSeq$  be a c-bar. Define $Q \subseteq
  \FSeq$ by $Q(a) \defeqiv a \cov_{\FBaire} P$. Then $P \subseteq Q$
  by $\eta$-inference, and $Q$ is inductive by $\digamma$-inference.
  Thus, by $\cBI$ we have $Q(\nil)$, i.e.\ $\nil \cov_{\FBaire} P$.
  
\medskip 

\noindent ($\ref{prop:cBICanon2} \Rightarrow \ref{prop:cBICanon1}$) 
   Assume \ref{prop:cBICanon2}.
     Let $P \subseteq \FSeq$ be a c-bar and $Q \subseteq
    \FSeq$ be an inductive subset such that $P \subseteq Q$.
    By the assumption, we have $\nil \cov_{\FBaire} P$, and since $P$ is monotone, 
    we have $\nil \bcov_{\FBaire} P$ by Lemma \ref{lem:ElimZeta}.
    Since $Q$ is closed under $\eta$ and $\digamma$-inferences
    (with respect to $P$), we have $Q(\nil)$.
\end{proof}

\begin{lemma}\label{lem:UnifCover}
   In formal Baire  space, we have
     $
     a \cov_{\FBaire} a[k]
     $
   for all $k \in \Nat$ where
   \[
     a[k] \defeql \left\{ a * b \mid b \in \FSeq \amp \lh{b} = k \right\}.
   \]
\end{lemma}
\begin{proof}
  By induction on $k$.
\end{proof}

\begin{theorem}\label{thm:EquivcBIUCb}
  $\FC \iff \cBI$.
\end{theorem}
\begin{proof}
($\Rightarrow$) 
Assume $\FC$. 
We prove  item \ref{prop:cBICanon2} of Proposition \ref{prop:cBICanon}.
Let $P$ be a c-bar,
and let $\delta \colon \FSeq
\to \Nat$ be a function such that 
  $
  P(a) \iff \left( \forall b \in \FSeq \right) \delta(a) =
  \delta(a * b)
  $
for all $a \in \FSeq$.
Define a pointwise continuous function $F \colon \PBaire \to \Nat$ 
as in \eqref{eq:FuncFromcBar}.
Then $F$ is formally continuous by $\FC$.  Thus, there exists a
formal topology map $r \subseteq \FSeq \times \Nat$ such that
$i_{\FNat} \circ  F = \Pt{r} \circ i_{\FBaire}$.
Let $ a \in r^{-}\Nat$, and let $n \in \Nat$ such that $a \mathrel{r} n$. 
Choose $k \in \Nat$ such that $|a| + k > n$.
Then, for each $b \in a[k]$, we have $F(b*\zero) = F(b*b'*\zero) = n$
for all $b' \in \FSeq$, which implies
$P(b)$.  Thus $a \cov_{\FBaire} P$ by Lemma
\ref{lem:UnifCover}, and hence $\nil \cov_{\FBaire} r^{-}\Nat
\cov_{\FBaire} P$.

\medskip 

\noindent($\Leftarrow$) Assume $\cBI$. Let $F \colon \PBaire \to \Nat$ be a
pointwise continuous function. Define a c-bar $P\subseteq \FSeq$ 
as in \eqref{eq:cBarFromFun}.  By $\cBI$, we have $\nil \cov_{\FBaire} P$.
Define a relation $r \subseteq \FSeq \times \Nat$ by
\[
  a \mathrel{r} n \defeqiv P(a) \amp F(a) = n.
\]
It is straightforward to show that $r$ is a formal topology map from  $\FBaire$ to
$\FNat$, and that $i_{\FNat} \circ  F = \Pt{r} \circ i_{\FBaire}$.
\end{proof}

By Proposition \ref{prop:scFTequivSC} and Theorem \ref{thm:EquivcBIUCb}, we conclude as follows.
\begin{theorem}
  If every strongly continuous function form Baire space to the
  natural numbers is formally continuous, then
  $\scFT$ implies $\cBI$.
\end{theorem}
Recall that Fan Theorem is a statement obtained from $\PiFT$ by omitting the
restriction on bars, and decidable Bar Induction is a statement similar to
$\cBI$ but formulated with respect to decidable bars.
Obviously, Fan Theorem implies $\scFT$  and $\cBI$ implies decidable
Bar Induction.

Fourman and Hyland \cite{FourmanHyland} constructed several sheaf models of constructive
analysis in which Fan Theorem holds but decidable Bar Induction
fails.\footnote{Note that these models are constructed in the
classical metatheory. In particular, Fan Theorem in the metatheory
plays a crucial role.}
In one of their models \cite[Theorem 3.8]{FourmanHyland},
there is a decidable, monotone, and inductive bar $P$ such that
$\neg P(\nil)$. Since $P$ is decidable and monotone, $P$ is a c-bar
with respect to its characteristic function $\chi_{P}\colon \FSeq \to
\Bin$.  Thus, we can define a pointwise continuous function $F \colon
\PBaire \to \Nat$ as in \eqref{eq:FuncFromcBar}. Since Fan Theorem
holds in this model, $F$ is strongly continuous.  If $F$ is formally
continuous, then we can derive $\nil \cov_{\FBaire} P$ as
in the proof of the direction ($\Rightarrow$) in Theorem
\ref{thm:EquivcBIUCb}. Since $P$ is monotone and inductive, we have
$P(\nil)$, which is a contradiction. Hence, $F$ is strongly continuous
but not formally continuous in this model.

\section{Brouwer-operations}\label{sec:BrouwerContinuity}
We show that formally continuous functions from Baire space to the natural
numbers are equivalent to functions induced by Brouwer-operations.
The latter notion is well known in intutionistic mathematics, and
plays an important role in the theory of choice sequences; see
Kreisel and Troelstra \cite[Section 3]{KreiselTroelstra} and
Troelstra and {van Dalen} \cite[Chapter 4 and Chapter 12]{ConstMathI,ConstMathII}.
\begin{definition}\label{def:NbF}
  A class $\nbf$ of \emph{Brouwer-operations} is inductively
  defined by the following two clauses:
  \begin{gather*}
    \frac{n \in \Nat}{\curry{a}{n+1} \in \nbf}, \qquad
    \frac{\gamma(\nil) = 0 \amp \left( \forall n \in \Nat \right)
      \gamma_{n} \in \nbf}{\gamma \in \nbf},
  \end{gather*}
  where $\gamma_{n} \defeql \curry{a}{\gamma(\seq{n}*a)}$ for each
  $n \in \Nat$.
  If $\gamma \in K$ is introduced by the second
  clause, we write $\sup_{n \in \Nat} \gamma_{n}$ for $\gamma$.
\end{definition}

\begin{remark}\label{rem:BOp}
  If the class $\BO$ of Brouwer ordinals form a set, then $K$ is a set. This can be
  seen as follows. First, we define a class $K^{*}$ by induction:
  \begin{gather*}
    \curry{a}{1} \in K^{*}, \qquad
    \frac{\gamma(\nil) = 0 \amp \left( \forall n \in \Nat \right)
      \gamma_{n} \in K^{*}}{\gamma \in K^{*}}.
  \end{gather*}
 Clearly, $K^{*}$ is isomorphic to $\BO$.
 Then, it is straightforward to show that $K$ is isomorphic to the set 
 \begin{equation*}
   \sum_{\gamma \in K^{*}}\Nat^{\ibar{\gamma}} \defeql \left\{ (\gamma,f) \mid
     \gamma \in K^{*} \amp f \colon \ibar{\gamma} \to \Nat \right\},
   \end{equation*}
  where
  \begin{equation}\label{eq:Bar}
    \ibar{\gamma} \defeql \left\{ a \in \FSeq \mid \gamma(a) > 0 \amp
    \left( \forall b \spfix a \right) \gamma(b) = 0 \right\}.
 \end{equation}
\end{remark}

Each Brouwer-operation $\gamma \in K$ is a neighbourhood function,
i.e.\ it has the following properties:
\begin{enumerate}
  \item 
  $
  \left( \forall \alpha \in \PBaire \right)
  \left( \exists n \in \Nat \right) \gamma(\overline{\alpha}n)
  > 0,
  $

  \item 
  $
  \left( \forall a,b \in \FSeq \right)\bigl[ \gamma(a) > 0 \rightarrow
  \gamma(a) = \gamma(a*b)  \bigr].
  $
\end{enumerate}
Hence, a Brouwer-operation $\gamma$ defines a continuous function
$F \colon \PBaire \to \Nat$ whose value at
$\alpha \in \PBaire$ is $\gamma(\overline{\alpha}n) \dotminus 1$, where
$\overline{\alpha}n$ is the shortest initial segment of $\alpha$
such that $\gamma(\overline{\alpha}n) > 0$; in other words $\overline{\alpha}n
\in \ibar{\gamma}$ where $\ibar{\gamma}$ is defined as in \eqref{eq:Bar}.
The function $F \colon \PBaire \to \Nat$ that arises in this way is called
realisable.
\begin{definition}\label{def:Rlz}
A function $F \colon \PBaire \to \Nat$ is \emph{realised} by a
Brouwer-operation $\gamma \in \nbf$ if
  \[
    \left( \forall \alpha \in \PBaire \right)\left( \exists n \in \Nat \right)
    \gamma(\iseg{\alpha}{n}) = F(\alpha) + 1.
  \]
  In this case, we write $F \rlz \gamma$.
  We say that a function $F \colon \PBaire \to \Nat$ is
  \emph{realisable} if it is realised by some Brouwer-operation.
\end{definition}
Note that each Brouwer-operation realises exactly one function,
but two different Brouwer-operations may realise the same
function.

\begin{proposition}\label{prop:Rlz}
  A function $F \colon \PBaire \to \Nat$ is realisable if and only if
  there exists a Brouwer-operation $\gamma \in \nbf$ such that
  \begin{equation}\label{eq:ConstOnBar}
    \left( \forall a \in \ibar{\gamma} \right)
    \left( \forall \alpha,\beta \in a \right) F(\alpha) = F(\beta).
  \end{equation}
\end{proposition}
\begin{proof}
  ($\Rightarrow$) Obvious.

\noindent($\Leftarrow$) For a function $F \colon \PBaire \to \Nat$ and a
  Brouwer-operation $\gamma \in \nbf$,
  write $\Phi(F,\gamma)$ if the condition \eqref{eq:ConstOnBar}
  holds.  We show that
  \begin{equation}\label{eq:ConstRlz}
    \left( \forall F \in \Nat^{(\PBaire)} \right) \Phi(F,\gamma) \implies 
    \left( \exists \gamma'  \in \nbf \right) F \rlz \gamma'
  \end{equation}
   for all $\gamma \in \nbf$ by induction on $\nbf$.
   
   \medskip 

  \noindent$\gamma = \curry{a}{n + 1}$ for some $n \in \Nat$: Let $F \colon
      \PBaire \to \Nat$ such that $\Phi(F,\gamma)$.  Since
      $\ibar{\gamma} = \left\{ \nil \right\}$, the function $F$ is
      constant. Thus $F \rlz \curry{a}{F(\zero) + 1}$.

   \medskip 

  \noindent$\gamma = \sup_{n \in \Nat} \gamma_{n}$:
        Let $F \colon \PBaire \to \Nat$ such that $\Phi(F,\gamma)$.
        Then, for each $n \in \Nat$, we have $\Phi(F_n, \gamma_{n})$,
        where $F_n \colon \PBaire \to \Nat$ is defined by 
          $
          F_n(\alpha) \defeql F(\seq{n}*\alpha).
          $
        By induction hypothesis and $\CC$, there exists a sequence $\left 
        ( \gamma'_{n} \right)_{n \in \Nat}$ of Brouwer-operations such
        that $F_{n} \rlz \gamma'_{n}$ for each $n \in \Nat$. Then $F
        \rlz \sup_{n \in \Nat}\gamma'_{n}$.
\end{proof}

\begin{definition}\label{def:Cov}
  Let $\Cov$ be a collection of pairs $(a,U) \in
  \FSeq \times \Pow{\FSeq}$ inductively  defined by the following clauses:
  \begin{gather*}
    \frac{a \in \FSeq}{ \{a\}  \in \Cov(a)}, \qquad
    \frac{\left( \forall n \in \Nat \right) U_{n} \in
    \Cov(a*\seq{n})}{\bigcup_{n \in \Nat} U_n \in \Cov(a)},
  \end{gather*}
  where $U \in \Cov(a) \defeqiv (a,U) \in \Cov$.
\end{definition}

The following lemma says that $\Cov$ is a \emph{set
presentation} of $\FBaire$. The result is not new; see e.g.\ van~den Berg and Moerdijk \cite[Proposition
B.4]{vandenBerg20121367}. However, our proof seems to be more direct and
worth noting.
\begin{lemma}\label{lem:CovPresentFBaire}
  For any $a \in \FSeq$ and $U \subseteq \FSeq$, we have
  \[ 
    a \cov_{\FBaire} U \iff \left(\exists V \in \Cov(a) \right) V
    \subseteq \ext{U}.
  \]
\end{lemma}
\begin{proof}
By Lemma \ref{lem:ElimZeta}, it suffices to show that 
\[
  a \bcov_{\FBaire} \ext{U} \iff
  \left(\exists V \in \Cov(a) \right) V \subseteq \ext{U}
\]
for all $a \in \FSeq$ and $U \subseteq \FSeq$. The proof is by
straightforward induction on $\bcov_{\FBaire}$ and $\Cov$
respectively. Note that the proof of the direction ($\Rightarrow$) requires
$\CC$ in the case where $a \bcov_{\FBaire} \ext{U}$ is derived by
$\digamma$-inference.
\end{proof}

\begin{proposition}\label{prop:FRep}
  A function $F \colon \PBaire \to \Nat$ is formally continuous if and only if
  there exists $U \in \Cov(\nil)$ such that
  \begin{equation}\label{eq:ConstOnCov}
    \left( \forall a \in U \right)
    \left( \forall \alpha, \beta\in a \right) F(\alpha) = F(\beta).
  \end{equation}
\end{proposition}
\begin{proof}
($\Rightarrow$) Suppose that $F$ is formally continuous. Then,
there exists a formal topology map $r \subseteq \FSeq \times \Nat$
such that $i_{\FNat} \circ F = \Pt{r} \circ i_{\FBaire}$.
Since $\nil \cov_{\FBaire} r^{-}\Nat$, there exists $U \in \Cov(\nil)$
such that $U \subseteq \ext{(r^{-}\Nat)}$ by Lemma
\ref{lem:CovPresentFBaire}.
Then, it is straightforward to show that $U$ satisfies the condition
\eqref{eq:ConstOnCov}.

\medskip

\noindent($\Leftarrow$) Suppose that there is $U \in \Cov(\nil)$ satisfying \eqref{eq:ConstOnCov}. Define a relation $r \subseteq \FSeq \times \Nat$ by
\[
  a \mathrel{r} n \defeqiv a \in U \amp F(a*\zero) = n.
\]
By Lemma \ref{lem:CovPresentFBaire}, we have $\nil
\cov_{\FBaire} r^{-}\Nat$. Moreover, the condition
\eqref{eq:ConstOnCov} ensures that $r$ satisfies the second condition of
formal topology map. Clearly we have $i_{\FNat} \circ F = \Pt{r} \circ i_{\FBaire}$.
\end{proof}

The following is a key lemma which relates formal continuity and
continuity with Brouwer-operations.
\begin{lemma}\label{lem:EqivSetPresentBOp}
For any subset $U \subseteq \FSeq$, we have
\begin{equation*}\label{eq:CovKEquiv}
  U \in \Cov(a) \iff \left( \exists \gamma \in K \right) 
  a * \ibar{\gamma} = U,
\end{equation*}
where $a * \ibar{\gamma} \defeql \left\{ a * b \mid b \in \ibar{\gamma} \right\}$. 
\end{lemma}
\begin{proof}
($\Rightarrow$)
By induction on $\Cov$.

\medskip

  \noindent$\left\{ a \right\} \in \Cov(a)$: Take $\gamma = \curry{a}{1}$.
  Then, $\ibar{\gamma} = \left\{ \nil \right\}$, so $a *
  \ibar{\gamma} = \left\{ a \right\}$.

\medskip

   \noindent$\displaystyle \frac{\left( \forall n \in \Nat \right) U_{n}
   \in \Cov(a*\seq{n})}{\bigcup_{n \in \Nat}U_n \in \Cov(a)}$: By
   induction hypothesis and $\CC$, there exists a sequence $\left( \gamma_{n}
   \right)_{n \in \Nat}$ of Brouwer-operations  such that
   $a*\seq{n}*\ibar{\gamma_{n}} = U_n$ for each $n \in \Nat$. Put
   $\gamma = \sup_{n \in \Nat}\gamma_{n}$.   Since
   $\ibar{\gamma} = \bigcup_{n \in \Nat}\seq{n}*\ibar{\gamma_{n}}$,
   we have $a * \ibar{\gamma} = \bigcup_{n \in \Nat}U_n$.

\medskip

\noindent($\Leftarrow$) It suffices to show that
\[
  \left( \forall a \in \FSeq \right) a * \ibar{\gamma} \in \Cov(a)
\]
for all $\gamma \in \nbf$, which is proved by induction on $\nbf$.
The argument is similar to the proof of the direction
($\Rightarrow$).  Note that $\CC$ is not required.
\end{proof}

\begin{remark}\label{rem:FBDefined}
Lemma \ref{lem:EqivSetPresentBOp} also shows that $\Cov$ is a set if
the Brouwer ordinals form a set; thus the lemma justifies the definition of
formal Baire space.
\end{remark}
\begin{corollary}\label{lem:AltRlz}
For any subset $U \subseteq \FSeq$, we have
\[
  U \in \Cov(\nil) \iff \left( \exists \gamma \in \nbf \right)
  \ibar{\gamma} = U.
\]
\end{corollary}

We are ready to state the main result of this section.
\begin{theorem}\label{thm:EqivRlzFRep}
  A function $F \colon \PBaire \to \Nat$ is realisable if and only if
  it is formally continuous.
\end{theorem}
\begin{proof}
Immediate from Proposition \ref{prop:Rlz}, Proposition \ref{prop:FRep}, and
Corollary \ref{lem:AltRlz}.
Note that the proof of ``only if'' part does not require
$\CC$.
\end{proof}

It is shown in \cite{KawaiUnifContBaire} that, under the assumption of
$\CC$, the statement $\cBI$ (see Section \ref{sec:StrongContinuity})
is equivalent to the following statement:
\begin{description}
  \item[\UCb] 
   Every pointwise continuous function $F \colon \PBaire \to \Nat$ is realisable.
\end{description}
This now becomes a corollary of Theorem \ref{thm:EquivcBIUCb} and Theorem \ref{thm:EqivRlzFRep}.

\section*{Acknowledgements}
The author thanks Giovanni Sambin and Milly Maietti for useful suggestions.
\bibliographystyle{abbrv}

 \newcommand{\noop}[1]{}

\end{document}